\documentclass[10pt,twoside]{siamltex}
\usepackage{amsfonts,epsfig}

\newtheorem{remark}[theorem]{Remark}

\newtheorem{example}[theorem]{Example}

\newcommand{\reals}{\mathbb{R}}

\newcommand{\ap}{\alpha}
\newcommand{\bt}{\beta}
\newcommand{\ld }{\lambda}

\newcommand{\sm}{\,\sigma\,}

\newcommand{\norm}[1]{\Vert #1 \Vert}

\newcommand{\HH}{H}

\newcommand{\N}{N}

\begin{document}

%  Leave these commented lines here 
% \input{elaheader-volx-xx.tex}
% \setcounter{page}{1}

% \renewcommand{\thefootnote}{\fnsymbol{footnote}}
% \renewcommand{\thefootnote}{\arabic{footnote}}
% \renewcommand{\theequation}{\thesection.\arabic{equation}}

\bibliographystyle{plain}
\title{Operator Connections and Borel Measures on the Unit Interval
        %\thanks{This
        %work was supported by the Society for Industrial and
        %Applied Mathematics, Philadelphia, Pennsylvania.}
        }

% The thanks line in the title should be filled in if there is
% any support acknowledgement for the overall work to be included
% This \thanks is also used for the received by date info, but
% authors are not expected to provide this.

\author{Pattrawut Chansangiam\thanks{Department of Mathematics, Faculty of Science,
    King Mongkut's Institute of Technology Ladkrabang,
    Bangkok 10520, THAILAND ({\tt kcpattra@kmitl.ac.th}). }
        \and Wicharn Lewkeeratiyutkul\thanks{Department of Mathematics and Computer Science, 
        Faculty of Science,
		Chulalongkorn University, Bangkok 10330, THAILAND  ({\tt Wicharn.L@chula.ac.th}).}}

\pagestyle{myheadings}
\markboth{P.\ Chansangiam, and W.\ Lewkeeratiyutkul}{Operator Connections and Borel Measures on the Unit Interval}
\maketitle

\begin{abstract}
A connection is a binary operation for positive operators satisfying
the monotonicity, the transformer inequality and the joint-continuity from above.
A mean is a normalized connection.
In this paper, we show that there is a one-to-one correspondence between connections and
finite Borel measures on the unit interval via a suitable integral representation.
Every mean can be regarded as an average of weighted harmonic means.
Moreover, we investigate decompositions of connections, means, symmetric connections and symmetric means.
\end{abstract}

\begin{keywords}
operator connection, operator mean, operator monotone function, Borel measure
\end{keywords}

\begin{AMS}
46G10, 47A63, 47A64
\end{AMS}

\section{Introduction}

A general theory of connections and means for positive operators
was given by Kubo and Ando \cite{Kubo-Ando}.
Let $B(\HH)$ be the von Neumann algebra of bounded linear operators on
a Hilbert space $\HH$. The set of positive operators on $\HH$ is denoted by $B(\HH)^+$.
For Hermitian operators $A,B \in B(\HH)$,
the partial order $A \leq B$ means that $B-A \in B(\HH)^+$.
A \emph{connection} is a binary operation $\sm$ on $B(\HH)^+$
such that for all positive operators $A,B,C,D$:
\begin{enumerate}
	\item[(M1)] \emph{monotonicity}: $A \leq C, B \leq D$ implies $A \sm B \leq C \sm D$
	\item[(M2)] \emph{transformer inequality}: $C(A \sm B)C \leq (CAC) \sm (CBC)$
	\item[(M3)] \emph{continuity from above}:  for $A_n,B_n \in B(\HH)^+$,
                if $A_n \downarrow A$ and $B_n \downarrow B$,
                 then $A_n \sm B_n \downarrow A \sm B$.
                 Here, $A_n \downarrow A$ indicates that $A_n$ is a decreasing sequence
                 and $A_n$ converges strongly to $A$.
\end{enumerate}
Two trivial examples are the left-trivial mean $(A,B) \mapsto A$ and the right-trivial mean $(A,B) \mapsto B$.
Typical examples of a connection are the sum $(A,B) \mapsto A+B$ and the parallel sum
\[
    A \,:\,B = (A^{-1}+B^{-1})^{-1}, \quad A,B>0,
\]
the latter being introduced by Anderson and Duffin \cite{Anderson-Duffin}.
A \emph{mean} is a connection $\sigma$ with normalized condition $I \sm I = I$ or, equivalently,
fixed-point property $A \sm A =A$ for all $A \geq 0$.
The class of Kubo-Ando means cover many well-known operator means
in practice, e.g.
	\begin{itemize}
		\item	$t$-weighted arithmetic means: $A \nabla_{t} B = (1-t)A + tB$
		\item	$t$-weighted geometric means:
    			$A \#_{t} B =  A^{1/2}
    			({A}^{-1/2} B  {A}^{-1/2})^{t} {A}^{1/2}$
		\item	$t$-weighted harmonic means: $A \,!_t\, B = [(1-t)A^{-1} + tB^{-1}]^{-1}$ % for $A,B>0$.
		\item	logarithmic mean: $(A,B) \mapsto A^{1/2}f(A^{-1/2}BA^{-1/2})A^{1/2}$ where $f: \reals^+ \to \reals^+$, $f(x)=(x-1)/\log{x}$.
	\end{itemize}
This axiomatic approach has many applications in operator inequalities
(e.g. \cite{Ando}, \cite{Furuta}, \cite{Mond_et_al.}),
operator equations (e.g. \cite{Anderson-Morley_matrix_equation}, \cite{Lim_matrix_eq})
and operator entropy (\cite{Fujii_entropy}).

A fundamental tool of Kubo-Ando theory of connections and means is the theory of
operator monotone functions.
This concept is introduced in \cite{Lowner}; see also \cite{Bhatia}, \cite{Hiai}, \cite{Hiai-Yanagi}.
A continuous real-valued function $f$ on an interval $I$ is called an \emph{operator monotone function} if,
for all Hermitian operators $A,B \in B(\HH)$ whose spectrums are contained in $I$ and for all Hilbert spaces $\HH$,
we have
\[
	A \leq B \Longrightarrow f(A) \leq f(B).
\]
A major core of Kubo-Ando theory is the interplay between connections,
operator monotone functions and Borel measures.
Note first that if $\sigma$ and $\eta$ are connections and $k_1,k_2 \in \reals^+ = [0,\infty)$, the binary operation
\[
		k_1 \sigma + k_2 \eta: (A,B) \mapsto k_1 (A \,\sigma\, B) + k_2 (A \,\eta\, B)
\]
is also a connection. This shows that the set of connections on $B(\HH)^+$ forms a cone.
Introduce a partial order $\leq$ on this cone by $\sigma \leq \eta$
if and only if $A \,\sigma\, B \leq A \,\eta\, B$ for all $A,B \geq 0$.
Equip the cone of operator monotone functions from $\reals^+$ to $\reals^+$ with the pointwise order, i.e.
$f \leq g$ means that $f(x) \leq g(x)$ for all $x \in \reals^+$.
The cone of finite Borel measures on a topological space is also equipped with the usual partial order.
In \cite{Kubo-Ando}, a connection $\sigma$ on $B(\HH)^+$ can be characterized
via operator monotone functions as follows: %\\

\begin{theorem}[\cite{Kubo-Ando}] \label{thm: Kubo-Ando f and sigma}
	Given a connection $\sigma$, there is a unique operator monotone function $f: \reals^+ \to \reals^+$ satisfying
				\[
    			f(x)I = I \sm (xI), \quad x \in \reals^+.
				\]
	Moreover, the map $\sigma \mapsto f$ is an affine order-isomorphism. 
	%Here, the order-isomorphism means that when $\sigma_i \mapsto f_i$ for $i=1,2$,
	%we have $\sigma_1 \leq \sigma_2$ if and only if $f_1 \leq f_2$.
\end{theorem} 

We call $f$ the \emph{representing function} of $\sigma$.
A connection also has a canonical representation with respect to a Borel measure as follows.

\begin{theorem}[\cite{Kubo-Ando}] 
	Given a connection $\sigma$, there is a unique finite Borel measure $\mu$ on $[0,\infty]$ such that
				\begin{equation}
    			A \sm B = \ap A + \bt B + \int_{(0,\infty)} \frac{\ld+1}{\ld} \{ (\ld A) \,:\,B\}\, d\mu(\ld),
    			\quad A,B \geq 0
    			\label{eq: formula of connection}
				\end{equation}
				where the integral is taken in the sense of Bochner, $\ap = \mu(\{0\})$ and $\bt = \mu(\{\infty\})$.
				Moreover, the map $\sigma \mapsto \mu$ is an affine isomorphism.
\end{theorem}
%\\

We call $\mu$ the \emph{representing measure} of $\sigma$.
In particular, every connection arises in form (\ref{eq: formula of connection}).
A connection is a mean if and only if its representing function is normalized (i.e. $f(1)= 1$)
or, equivalently, its representing measure is a probability measure.

In this paper, we characterize connections, means, symmetric connections and symmetric means in terms of Borel
measures on the unit interval.
The main result of this paper is the following:
%\\

\begin{theorem} \label{thm: 1-1 conn and measure}
Given a finite Borel measure $\mu$ on $[0,1]$, the binary operation
	\begin{equation}
		A \sm B = \int_{[0,1]} A \,!_t\, B \,d \mu(t), \quad A,B \geq 0
		\label{eq: int rep connection}
	\end{equation}
	is a connection on $B(\HH)^+$.
	Moreover, the map $\mu \mapsto \sigma$ is bijective, affine and order-preserving,
	in which case the representing function of $\sigma$ is given by
	\begin{equation}
		f(x) = \int_{[0,1]} 1 \,!_t\, x \,d \mu(t), \quad x \geq 0. \label{int rep of OMF}
	\end{equation}
\end{theorem}

This theorem states that there is an affine isomorphism between connections
and finite Borel measures on $[0,1]$ via a suitable form of integral representation.
Moreover, a connection $\sigma$ is a mean if and only if $\mu$ is a probability measure.
Hence every mean can be regarded as an average of weighted harmonic means.
The weighted harmonic means form a building block for general connections.
A surprising example is that the dual of the logarithmic mean
can be expressed as the integral of weighted harmonic means with respect to Lebesgue measure on $[0,1]$.
%\begin{equation}
%		A \,\eta\, B = \int_{[0,1]} A \,!_t\, B \,dm(t), \quad A,B \geq 0
%\end{equation}
%where $m$ is Lebesgue measure.
Recall that a symmetric connection is a connection $\sigma$ such that $A \sigma B = B \sigma A$ for all $A,B \geq 0$.
It follows that every symmetric connection admits an integral representation
\[
		A \sm B = \frac{1}{2} \int_{[0,1]} (A \,!_t\, B)  +  (B \,!_t\, A) \,d \mu(t),
		\quad A,B \geq 0
\]
for some finite Borel measure $\mu$ invariant under
$\Theta: [0,1] \to [0,1]$, $t \mapsto 1-t$.
The integral representation (\ref{eq: int rep connection}) also has advantages in treating
decompositions of connections.
It is shown that a connection $\sigma$ can be written as
\[
	\sigma = \sigma_{ac} + \sigma_{sd} + \sigma_{sc}
\]
where $\sigma_{ac}$, $\sigma_{sd}$ and $\sigma_{sc}$ are connections.
The ``singular discrete part" $\sigma_{sd}$ is a countable sum of weighted harmonic means with nonnegative coefficients.
The ``absolutely continuous part" $\sigma_{ac}$ has an integral representation
with respect to Lebesgue measure $m$ on $[0,1]$.
The ``singular continuous part" $\sigma_{sc}$ has an integral representation with respect to a continuous measure
mutually singular to $m$.
%Furthermore, we decompose means, symmetric connections and
%symmetric means.

Here is the outline of the paper.
In Section 2, we prove Theorem \ref{thm: 1-1 conn and measure} and its consequences.
The integral representations of well-known operator connections with respect to finite Borel measures on $[0,1]$
are also given.
Decompositions of connections, means, symmetric connections and symmetric means are discussed in Section 3.

\section{Connections and Borel measures on the unit interval}

In this section, we investigate the relationship between connections and Borel measures on $[0,1]$.
Our main result states that there is an affine correspondence between connections and finite Borel measures on $[0,1]$
via a suitable integral representation. To prove this fact, recall the following:
%\\

\begin{lemma}[\cite{Lowner}] \label{thm: Lowner}
    A continuous function $f: \reals^+ \to \reals^+$ is operator monotone if and only if there is a unique
    finite Borel measure $\nu$ on $[0,\infty]$ such that
\begin{equation}
		f(x) = \int_{[0,\infty]} \frac{x(\ld+1)}{x+\ld} \,d \nu(\ld), \quad x \geq 0. \label{eq: int rep of OMF}
\end{equation}
\end{lemma}

%\begin{proof}[Proof of Theorem \ref{thm: 1-1 conn and measure}.]
\emph{Proof of Theorem \ref{thm: 1-1 conn and measure}.}
Since the family $\{A \,!_t\, B\}_{t \in [0,1]}$ is uniformly bounded by
$\max\{\norm{A}, \norm{B}\}  \mu([0,1])$, the binary operation
\[
	A \,\sigma\, B = \int_{[0,1]} A\,!_t\,B\, d\mu(t), \quad A,B \geq 0
\]
is well-defined.
The monotonicity (M1) and the transformer inequality (M2) follow from passing these properties
of weighted harmonic means through the integral.
To show (M3), let $A_n \downarrow A$ and $B_n \downarrow B$.
Then $A_n \, !_t \, B_n \downarrow  A \, !_t \, B$
for $t \in [0,1]$ by the joint-continuity of weighted harmonic means.
We obtain from the dominated convergence theorem that for each $\xi \in H$
\begin{eqnarray*}
    \lim_{n \to \infty} \langle (A_n \sm B_n) \xi, \xi \rangle
        &=&  \lim_{n \to \infty} \langle \int A_n \, !_t \, B_n\, d\mu(t) \xi, \xi \rangle   \\
        &=&  \lim_{n \to \infty} \int \langle (A_n \, !_t \, B_n) \xi, \xi \rangle \, d\mu(t) \\
        &=&  \int  \langle (A \, !_t \, B) \xi, \xi \rangle \, d\mu(t) \\
        &=&  \langle \int   A \, !_t \, B \, d\mu(t)\xi, \xi \rangle, %   \\
        %&=  \langle (A \sm B)\xi, \xi \rangle.
\end{eqnarray*}
i.e. $A_n \sm B_n \downarrow A \sm B$.
Thus, we have established a well-defined map $\mu \mapsto \sigma$.
For injectivity of this map, let $\mu_1$ and $\mu_2$ be finite Borel measures on $[0,1]$ such that
\[
	\int_{[0,1]} A\,!_t\,B\, d\mu_1(t) = \int_{[0,1]} A\,!_t\,B\, d\mu_2(t), \quad A,B \geq 0.
\]
For each $x \geq 0$, define
\[
    f_i (x) = \int_{[0,1]} 1 \,!_t\, x  \,d\mu_i(t)
            = \int_{[0,\infty]} \frac{x(\ld+1)}{x+\ld} \,d\mu_i \Psi(\ld)\quad i=1,2
\]
where $\Psi: [0,\infty] \to [0,1]$, $t \mapsto t/(t+1)$ and the measure $\mu_i \Psi$ is defined 
for each Borel set $E$  by $E \mapsto \mu_i (\Psi(E))$.
Then for each $x \geq 0$,  we have $f_1(x) = f_2(x)$ by setting $A=I$ and $B=xI$.
Theorem \ref{thm: Lowner} implies that $\mu_1 = \mu_2$.

For surjectivity of this map, consider a connection $\sigma$.
By Theorem \ref{thm: Kubo-Ando f and sigma}, there is a unique operator monotone function  $f:\reals^+ \to \reals^+$
such that $f(x) I  =  I  \sm (x I )$ for $x \geq 0$.
By Theorem \ref{thm: Lowner}, there is a finite Borel measure $\nu$ on $[0,\infty]$ such that
(\ref{eq: int rep of OMF}) holds.
%\begin{equation}
%		f(x) = \int_{[0,\infty]} \frac{x(\ld+1)}{x+\ld} \,d \nu(\ld), \quad x \geq 0.
%\end{equation}
Define a finite Borel measure $\mu$ on $[0,1]$ by $\mu = \nu \Psi^{-1}$.
Consider a connection $\eta$ defined by
\[
	A \,\eta\, B = \int_{[0,1]} A\,!_t\,B\, d\mu(t) , \quad A,B \geq 0.
\]
Then $I \,\eta\, (xI) = f(x)I$ for $x \geq 0$.
Theorem \ref{thm: Kubo-Ando f and sigma} implies that $\eta = \sigma$.
Hence the map $\mu \mapsto \sigma$ is surjective.
This map is affine since $\mu$ is finite.
It is easy to see that this map is order-preserving.
The proof of Theorem \ref{thm: 1-1 conn and measure} is therefore complete.
%\end{proof}

From Theorem \ref{thm: 1-1 conn and measure}, every connection $\sigma$ takes the form
\[
	A \,\sigma\, B = \int_{[0,1]} A\,!_t\,B\, d\mu(t), \quad A,B \geq 0
\]
for a unique finite Borel measure $\mu$ on $[0,1]$.
%The integral representation \eqref{eq: int rep connection} indicates that every connection can be viewed
%as a weighed series connection of weighted harmonic means.
Note that the $0$-weighted harmonic mean and the $1$-weighted harmonic mean are the left-trivial mean
and the right-trivial mean, respectively.
We call the measure $\mu$ corresponding to a connection $\sigma$
the \emph{associated measure} of $\sigma$.
The relationship between the associated measure $\mu$ and the representing measure $\nu$ of
a connection is given by $\mu = \nu \Psi^{-1}$.

\begin{example} {\rm
%\textbf{Example 1.} 
\label{ex: assoc measure of connection}
	The associated measure of the $t$-weighted harmonic mean $!_t$ is the Dirac measure $\delta_t$ at $t$.
	In particular, the associated measures of the left-trivial mean and the right-trivial mean
	are $\delta_0$ and $\delta_1$, respectively.
    By affinity of the map $\mu \mapsto \sigma$, the associated measures of the sum and the parallel sum are given by $\delta_0 + \delta_1$ and $\frac{1}{2}\delta_{1/2}$, respectively.
	Similarly, the $\ap$-weighted arithmetic mean has the associated measure
  $(1-\ap) \delta_0  + \ap \delta_1$.
  More generally, the measure
  $
    \sum_{i=1}^n a_i \, \delta_{t_i}
  $,
  where $t_i \in [0,1]$ and $a_i\geq 0$, is associated to the connection $\sum_{i=1}^n a_i \, !_{t_i}$.
  In particular, the probability measure $ (1-\ap) \delta_t + \ap \delta_s$, when $\ap,t,s \in [0,1]$,
  is associated to the $\ap$-weighted arithmetic mean between
  the $t$-weighted harmonic mean and the $s$-weighted harmonic mean.
  }
\end{example}
%\\

\begin{example} \label{ex: assoc measure of weight geometric}	{\rm 
Consider the associated measure of the $\ap$-weighted geometric mean $\#_{\ap}$ for $0<\ap<1$.
From contour integrals, its representing function is given by
\[
    x^{\ap} =  \int_{[0,\infty]} \frac{x \ld^{\ap -1}}{x+\ld} \cdot \frac{\sin \ap \pi}{\pi}  \, dm(\ld).
\]
It follows that the associated measure of $\#_{\ap}$ is %given by
$\displaystyle{%\begin{equation}
    d\mu(t) %&= \frac{\sin \ap \pi}{\pi} \cdot  (\frac{t}{1-t})^{\ap-1} \cdot \frac{1}{ \frac{t}{1-t} +1}
            %    \,dm (\frac{t}{1-t}) \\
        %&= \frac{\sin \ap \pi}{\pi} \cdot  \frac{(1-t)^{1-\ap}}{t^{1-\ap}} \cdot (1-t) \cdot \frac{1}{(1-t)^2}\,dt \\
        = \frac{\sin \ap \pi}{\pi} \cdot  \frac{1}{t^{1-\ap} (1-t)^{\ap}} \,dt.
}
$
}
\end{example}

\begin{example} \label{ex: assoc measure of LM}
{\rm
	Let us compute the associated measure of the logarithmic mean.
	Recall that the representing function of this mean
	is the operator monotone function
	$
		f(x) = (x-1) / \log{x}
	$.
	Then, by Example \ref{ex: assoc measure of weight geometric}, we have
	\begin{eqnarray*}
		f(x) &=& \int_0^1 x^{\ld} \,d \ld
			\, = \, \int_0^1 \int_0^1 \frac{\sin{\ld \pi}}{\pi} \cdot
                    \frac{1 \,!_{t}\, x}{t^{1- \ld}(1-t)^{\ld}} \,dt \,d \ld \\
			&=& \int_0^1 (1 \,!_{t}\, x) \int_0^1 \frac{\sin{\ld \pi}}{\pi t^{1- \ld}(1-t)^{\ld}}
			\,d\ld \,dt.
	\end{eqnarray*}
	Hence the associated measure has density $g$ given  by
	\[
		g(t) = \int_0^1 \frac{\sin{\ld \pi}}{\pi t^{1-\ld} (1-t)^{\ld}} \,d\ld
				 = \frac{1}{t(1-t) \left(\pi^2 + \log^2 (\frac{t}{1-t}) \right) }.
	\]
}
\end{example}

\begin{example} \label{ex: assoc measure of dual of LM}
{\rm
    Consider the dual of the logarithmic mean defined by
    \[
        A \,\eta\, B = \{ LM(B^{-1}, A^{-1}) \}^{-1},
    \]
    where $LM$ denotes the logarithmic mean.
    It turns out that its associated measure is Lebesgue measure on $[0,1]$.
    Indeed, its representing function is given by $x \mapsto \frac{x}{x-1} \log{x}$.
    We thus have the integral representation
    \[
        A \,\eta\, B  =  \int_{[0,1]} A \,!_t\, B \, dt.
    \]
}
\end{example}

\begin{remark} {\rm
	Even though the map $\mu \mapsto \sigma$ is order-preserving,
	the inverse map $\sigma \mapsto \mu$ is not order-preserving in general.
	For example, the associated measures of the harmonic mean $!_{1/2}$
	and the arithmetic mean $\nabla_{1/2} = (!_0 + !_1)/2$
	are given by $\delta_{1/2}$ and $(\delta_0 + \delta_1)/2$, respectively.
	We have $!_{1/2} \leq \nabla_{1/2}$ but $\delta_{1/2} \not\leq (\delta_0 + \delta_1)/2$.
}
\end{remark}
%\\

\begin{corollary} \label{cor: mean iff prob meas}
	A connection is a mean if and only if its associated measure is a probability measure.
\end{corollary}
\begin{proof}
	Let $\sigma$ be a connection with representing function $f$ and associated measure $\mu$.
	Recall that $\sigma$ is a mean if and only if $f(1)=1$.
	From the integral representation (\ref{eq: int rep connection}),
    we have $f(1)I = I \,\sigma\, I = \mu([0,1])I$.
\end{proof}
%\\
%This corollary states that every mean can be viewed as an average of weighted harmonic means.
%Hence weighted harmonic means form building blocks for general operator means.

\begin{corollary}
	There is a one-to-one correspondence between means and probability Borel measures on the unit interval.
	In fact, every mean takes the form
	\[
		A \sm B = \int_{[0,1]} A \,!_t\, B \,d \mu(t), \quad A,B \geq 0
	\]
	for a unique probability Borel measure $\mu$ on $[0,1]$.
\end{corollary}
%\\

Hence every mean can be regarded as an average of special means, namely, weighted harmonic means.
The weighted harmonic means form building blocks for general means.
%\\

\begin{corollary}
	The weighted harmonic means are the extreme points of the convex set of means.
\end{corollary}
\begin{proof}
	Use Theorem \ref{thm: 1-1 conn and measure}, Corollary \ref{cor: mean iff prob meas} and
	the fact that the Dirac measures are the extreme points of the convex set
	of probability Borel measures on $[0,1]$.
\end{proof}
%\\

The transpose of a connection $\sigma$ is the connection $(A,B) \mapsto B \,\sigma\, A$.
Hence, a connection is symmetric if and only if it coincides with its transpose.
It was shown in \cite{Kubo-Ando} that if $f$ is the representing function of $\sigma$, then
the representing function of the transpose of $\sigma$ is given by $x \mapsto  x f(1/x)$ for $x>0$.
%\\

\begin{theorem} \label{thm: rep func and rep measure of transp}
	Let $\sigma$ be a connection with associated measure $\mu$. Then
	\begin{itemize}
	\item	the representing function of the transpose of $\sigma$ is given by
				\begin{equation}
					x \mapsto \int_{[0,1]} x \,!_t\, 1 \,d \mu(t), \quad x \geq 0; \label{eq: rep func of transp of con}
				\end{equation}
	\item	the associated measure of the transpose of $\sigma$ is given by
				$\mu \Theta^{-1} =\mu \Theta $ where $\Theta  : [0,1] \to [0, 1], t \mapsto 1-t$.
	\end{itemize}
\end{theorem}	
\begin{proof}
	The set function $\mu\Theta$ is clearly a finite Borel measure on $[0,1]$.
	Since the representing function $f$ of $\sigma$ is given by (\ref{int rep of OMF}),
	the representing function of the transpose of $\sigma$ is given by
	\[
		x f(\frac{1}{x})
		= x \int_{[0,1]} 1 \,!_t\, \frac{1}{x}\, d\mu(t)
		= \int_{[0,1]} x \,!_t\, 1\, d\mu(t) %\\
		%&= \int_{[0,1]} x \,!_{1-t} \, 1 \, d\mu(1-t) = \int_{[0,1]} 1 \,!_{t}\, x\, d\mu\Theta(t)
	\]
	for each $x>0$. By continuity, the representing function of the transpose of $\sigma$ is
	given by (\ref{eq: rep func of transp of con}).
	%The correspondence between connections and Borel measures in
	By changing the variable $t \mapsto 1-t$ and Theorem \ref{thm: 1-1 conn and measure},
	we obtain that the associated measure of the transpose of $\sigma$ is $\mu \Theta$.
\end{proof}
%\\

We say that a Borel measure $\mu$ on $[0,1]$ is \emph{symmetric} if $\mu \Theta = \mu$.
%\\

\begin{corollary}	\label{cor: symmetric conn}
	There is a one-to-one correspondence between symmetric connections and finite symmetric Borel measures
	on $[0,1]$ via the integral representation
	\[
		A \sm B = \frac{1}{2} \int_{[0,1]} (A \,!_t\, B)  +  (B \,!_t\, A) \,d \mu(t),
		\quad A,B \geq 0 \label{int rep of symmetric con}
	\]
	The representing function of the connection $\sigma$ in (\ref{int rep of symmetric con})
	can be written by
	\begin{equation}
		f(x) = \frac{1}{2} \int_{[0,1]} (1 \,!_t\, x) + (x \,!_t\, 1 ) \,d \mu(t), \quad x \geq 0
	\end{equation}
	and its associated measure is $\mu$.
	In particular, a connection is symmetric if and only if its associated measure is symmetric.
\end{corollary}
\begin{proof}
	It follows from Theorems \ref{thm: 1-1 conn and measure} and \ref{thm: rep func and rep measure of transp}.
\end{proof}
%\\

\begin{corollary}
	There is a one-to-one correspondence between symmetric means and
	probability symmetric Borel measures on the unit interval.
	In fact, every symmetric mean takes the form
	\[
		A \sm B = \frac{1}{2} \int_{[0,1]} (A \,!_t\, B) + (B \,!_t\, A) \,d \mu(t), \quad A,B \geq 0
	\]
	for a unique probability symmetric Borel measure $\mu$ on $[0,1]$.
\end{corollary}
\begin{proof}
	Use Corollaries \ref{cor: mean iff prob meas} and \ref{cor: symmetric conn}.
\end{proof}

\section{Decompositions of connections}

This section deals with decompositions of connections, means, symmetric connections and symmetric means.
%\\

\begin{theorem} \label{thm: connection decom}
    Let $\sigma$ be a connection on $B(\HH)^+$.
Then there is a unique triple $(\sigma_{ac},\sigma_{sc},\sigma_{sd})$ of connections on $B(\HH)^+$ such that
\[
    \sigma = \sigma_{ac} + \sigma_{sc} + \sigma_{sd} \label{eq: sigma decomposition}
\]
and
\begin{enumerate}
    \item[(i)] $\sigma_{sd}$ is a countable sum of weighted harmonic means with nonnegative coefficients,
    				i.e. there are a countable set $D \subseteq [0,1]$ and a family
            $\{a_{t}\}_{t \in D} \subseteq \reals^+$ such that $\sum_{t \in D} a_{t} < \infty$ and
            \[
                \sigma_{sd} = \sum_{t \in D} a_{t} \,!_{t},
            \]
            i.e. for each $A,B \geq 0$, $A \,\sigma_{sd}\, B = \sum_{t \in D} a_{t} (A \,!_{t}B) $ and the
            series converges in the norm topology;
    \item[(ii)]   there is a (unique $m$-a.e.) integrable function $g: [0,1] \to \reals^+$ such that
            \[
                A \,\sigma_{ac} \, B = \int_{[0,1]} g(t) ( A \,!_{t}\, B) \,dm(t), \quad A,B \geq 0;
                \label{eq: sigma ac}
            \]
    \item[(iii)] its associated measure of $\sigma_{sc}$ is continuous and mutually singular to $m$.
\end{enumerate}
Moreover, the representing functions of $\sigma_{ac}, \sigma_{sd}$ and $\sigma_{sc}$ are given respectively by
\begin{eqnarray}
	f_{sd} (x) &=& \int_{[0,1]} 1\,!_t\, x\, d\mu_{sd} =  \sum_{t \in D} a_{t}(1 \,!_{t}\,x) \label{formula of f_sd}\\
	f_{ac} (x) &=& \int_{[0,1]} 1\,!_t\, x\, d\mu_{ac},  \label{formula of f_ac}\\
	f_{sc} (x) &=& \int_{[0,1]} 1\,!_t\, x\, d\mu_{sc}	\label{formula of f_sc}
\end{eqnarray}
and the associated measure of $\sigma_{sd}$ is given by $\sum_{t \in D} a_{t} \,\delta_{t}$.
\end{theorem}

\begin{proof}
%Let $f$ be the representing function of $\sigma$.
Let $\mu$ be the associated measure of $\sigma$.
Then there is a unique triple $(\mu_{ac},\mu_{sc},\mu_{sd})$
of finite Borel measures on $[0,1]$ such that $\mu= \mu_{ac} + \mu_{sc} + \mu_{sd}$ where
$\mu_{sd}$ is a discrete measure, $\mu_{ac}$ is absolutely continuous with respect to $m$  and
$\mu_{sc}$ is a continuous measure mutually singular to $m$.
Define $\sigma_{ac}, \sigma_{sc}, \sigma_{sd}$ to be the connections associated to 
$\mu_{ac},\mu_{sc},\mu_{sd}$, respectively. Then
\[
    %A \,\sigma_{ac}\, B &=& \int_{[0,1]} A \,!_{t}\,B \, d\mu_{ac}(t), \\ %\label{eq: sigma a}\\
    %A \,\sigma_{sc}\, B &=& \int_{[0,1]} A \,!_{t}\,B \, d\mu_{sc}(t), \\ %\label{eq: sigma sc} \\
    A \,\sigma_{sd}\, B = \int_{[0,1])} A \,!_{t}\,B \, d\mu_{sd}(t)
        = \sum_{t \in D} a_{t} (A \,\,!_{t}\,B), \quad  A,B \geq 0,%\label{eq: sigma sd}
\]
where the series $\sum_{t \in D} a_{t} (A \,\,!_{t}\,B)$ converges in norm.
Indeed, the fact that, for each $n<m$ in $\N$ and $t_i \in [0,1]$,
\begin{eqnarray*}
	 \bigg \Vert \sum_{i=1}^n a_{t_i} (A \,!_{t_i} \, B) - \sum_{i=1}^m a_{t_i} (A \,!_{t_i} \, B)  \bigg \Vert
		%&= \bigg \Vert \sum_{i=n}^m a_{t_i} (A \,!_{t_i} \, B) \bigg \Vert \\
		&\leq & \sum_{i=n+1}^m a_{t_i} \norm{A \,!_{t_i} \, B} \\
		&\leq & \sum_{i=n+1}^m a_{t_i} (\norm{A} \,!_{t_i} \, \norm{B}) \\
		&\leq & \sum_{i=n+1}^m a_{t_i} \max\{\norm{A}, \norm{B}\}
\end{eqnarray*}
together with the condition $\sum_{i=1}^{\infty} a_{t_i} < \infty$ implies the convergence of the series 
$\sum_{i=1}^{\infty} a_{t_i} (A \,!_{t_i}\,B)$.
Using Theorem \ref{thm: Kubo-Ando f and sigma}, the representing functions of $\sigma_{sd}, \sigma_{ac}, \sigma_{sc}$
are given by $f_{sd}, f_{ac}, f_{sc}$ in (\ref{formula of f_sd}), (\ref{formula of f_ac}), 
(\ref{formula of f_sc}), respectively.
The condition (i) comes from the fact that the associated measure of $!_t$ is $\delta_t$ for each $t \in [0,1]$ in Example
\ref{ex: assoc measure of connection}.
The condition (ii) follows from Radon-Nikodym theorem.
The uniqueness of such decomposition is obtained from the uniqueness of the decomposition of measures.
%The one-to-one correspondence between operator monotone functions on $\reals^+$ and finite Borel measures on $[0,1]$
%shows that the associated measures of $\sigma_{ac}, \sigma_{sd}$ and $\sigma_{sc}$ are
%given by $\mu_{ac}, \mu_{sd}$ and $\mu_{sc}$, respectively.
%Since $f = f_{ac}+f_{sd}+f_{sc}$, we have $\sigma = \sigma_{ac} + \sigma_{sd} + \sigma_{sc}$.
%The rest of theorem comes from Theorem \ref{thm: 1-1 conn and measure} and
%Note that the series in \eqref{formula of f_sd} converges since $\sum_{t \in D} a_{t} < \infty$.
\end{proof}
%\\

This theorem says that every connection $\sigma$ consists of three parts.
The ``singular discrete part" $\sigma_{sd}$ is a countable sum of weighted harmonic means.
Such type of connections include the weighted arithmetic means, the sum and the parallel sum.
The ``absolutely continuous part" $\sigma_{ac}$ arises as an integral
with respect to Lebesgue measure, given by the formula (\ref{eq: sigma ac}).
Examples \ref{ex: assoc measure of weight geometric}, \ref{ex: assoc measure of LM} and
\ref{ex: assoc measure of dual of LM} show that
the weighted geometric means, the logarithmic mean and its dual are typical examples of such connections.
The ``singular continuous part" $\sigma_{sc}$ has an integral representation with respect to a continuous measure
mutually singular to Lebesgue measure.
%Examples of such measures correspond to nonconstant continuous functions $F: \R \to \C$
%such that $F$ is of bounded variation, $F(-\infty)=0$ and $F'=0$ almost everywhere.
%One such function is the Cantor function (this gives rise to the Cantor measure).
Hence (aside singular continuous part) this theorem gives an explicit description of general connections.
%\\

\begin{proposition} \label{prop: properties of sigma ac}
Consider the connection $\sigma_{ac}$ defined by (\ref{eq: sigma ac}). Then
    \begin{itemize}
        \item   it is a mean if and only if the average of the density function $g$ is $1$, i.e.
                $
	               \int_0^1 g(t)\,dt = 1.
                $
                %i.e. the average of $g$ on $[0,1]$ is $1$.
        \item   it is a symmetric connection if and only if $g \circ \Theta = g$.
    \end{itemize}
\end{proposition}
\begin{proof}
    Use Corollaries \ref{cor: mean iff prob meas} and \ref{cor: symmetric conn}.
\end{proof}
%\\

We say that a density function $g:[0,1] \to \reals^+$ is \emph{symmetric} if $g \circ \Theta = g$.
Next, we decompose a mean as a convex combination of means.
%\\

\begin{corollary} \label{cor: mean decom}
    Let $\sigma$ be a mean on $B(\HH)^+$.
Then there are unique triples $(\widetilde{\sigma_{ac}},\widetilde{\sigma_{sc}},\widetilde{\sigma_{sd}})$
of means or zero connections on $B(\HH)^+$ and $(k_{ac}, k_{sc}, k_{sd})$ of real numbers in $[0,1]$ such that
\[
    \sigma = k_{ac} \widetilde{\sigma_{ac}} + k_{sc} \widetilde{\sigma_{sc}} + k_{sd} \widetilde{\sigma_{sd}},
    \quad k_{ac} + k_{sc} + k_{sd} = 1
    \label{eq: mean decomposition}
\]
and
\begin{enumerate}
    \item[(i)] there are a countable set $D \subseteq [0,1]$ and
            a family $\{a_{t}\}_{t \in D} \subseteq \reals^+$ such that $\sum_{t \in D} a_{t} =1$ and
            $ %\begin{equation}
                \widetilde{\sigma_{sd}} = \sum_{t \in D} a_{t} \,!_{t};
            $%\end{equation}
    \item[(ii)]   there is a (unique $m$-a.e.) integrable function $g: [0,1] \to \reals^+$ with average $1$ such that
    				%$\int_{[0,1]} g(t) \,dm(t) = 1$ and
            \[
                A \,\widetilde{\sigma_{ac}} \, B = \int_{[0,1]} g(t) ( A \,!_{t}\, B) \,dm(t), \quad A,B \geq 0;
            \]
    \item[(iii)] its associated measure of $\widetilde{\sigma_{sc}}$ is continuous and mutually singular to $m$.
\end{enumerate}
\end{corollary}
\begin{proof}
	Let $\mu$ be the associated probability measure of $\sigma = \sigma_{ac} + \sigma_{sd} + \sigma_{sc}$
	and write $\mu = \mu_{ac} + \mu_{sd} + \mu_{sc}$.
	Suppose that $\mu_{ac}$, $\mu_{sd}$ and $\mu_{sc}$ are nonzero.
	Then
	\[
		\mu = \mu_{ac} ([0,1]) \frac{\mu_{ac}}{\mu_{ac} ([0,1])} + \mu_{sd} ([0,1]) \frac{\mu_{sd}}{\mu_{sd} ([0,1])}
				+ \mu_{sc} ([0,1]) \frac{\mu_{sc}}{\mu_{sc} ([0,1])}.
	\]
	Set
	\begin{eqnarray*}
		\widetilde{\mu_{a}} &=& \frac{\mu_{ac}}{\mu_{ac} ([0,1])}, \quad \widetilde{\mu_{sd}}
        = \frac{\mu_{sd}}{\mu_{sd} ([0,1])},
			\quad \widetilde{\mu_{sc}} = \frac{\mu_{sc}}{\mu_{sc} ([0,1])}, \\
		k_{ac} &=& \mu_{ac} ([0,1]), \quad k_{sd} = \mu_{sd} ([0,1]), \quad k_{sc} = \mu_{sc} ([0,1]).
	\end{eqnarray*}
	%Then $\widetilde{\mu_{a}},\widetilde{\mu_{sc}},\widetilde{\mu_{sd}}$
	%are probability measures and $k_{ac}+k_{sd}+k_{sc} = 1$.
	Define $\widetilde{\sigma_{ac}}, \widetilde{\sigma_{sd}}, \widetilde{\sigma_{sc}}$ to be the means corresponding to
	the measures $\widetilde{\mu_{ac}}, \widetilde{\mu_{sd}}, \widetilde{\mu_{sc}}$, respectively.
	Now, apply Theorem \ref{thm: connection decom} and Proposition
    \ref{prop: properties of sigma ac}.
\end{proof}
%\\

We can decompose a symmetric connection
as a nonnegative linear combination of symmetric connections as follows:
%\\

\begin{corollary} \label{cor: symmetric connection decom}
    Let $\sigma$ be a symmetric connection on $B(\HH)^+$.
Then there is a unique triple $(\sigma_{ac},\sigma_{sc},\sigma_{sd})$ of symmetric connections on $B(\HH)^+$ such that
\[
    \sigma = \sigma_{ac} + \sigma_{sc} + \sigma_{sd} \label{eq: sym con decomposition}
\]
and
\begin{enumerate}
    \item[(i)] %$\sigma_{sd}$ is a countable sum of weighted harmonic means with nonnegative symmetric coefficients, i.e.
    there are a countable set $D \subseteq [0,1]$ and a family $\{a_{t}\}_{t \in D} \subseteq \reals^+$ such that
    $a_t = a_{1-t}$ for all $t \in D$, $\sum_{t\in D} a_t < \infty$ and
            $%\begin{equation}
                \sigma_{sd} = \sum_{t \in D} a_{t} \,!_{t};
            $%\end{equation}
    \item[(ii)]   there is a (unique $m$-a.e.) symmetric integrable function $g: [0,1] \to \reals^+$ such that
            \[
                A \,\sigma_{ac} \, B = \frac{1}{2} \int_{[0,1]} g(t)
                ( A \,!_{t}\, B + B \,!_{t}\, A) \,dm(t), \quad A,B \geq 0;
            \]
    \item[(iii)] its associated measure of $\sigma_{sc}$ is continuous and mutually singular to $m$.
\end{enumerate}
\end{corollary}
\begin{proof}
	Let $\mu$ be the associated measure of $\sigma$ and decompose
	$\mu = \mu_{ac} + \mu_{sd} + \mu_{sc}$ where
	$\mu_{ac} \ll m$, $\mu_{sd}$ is a discrete measure, $\mu_{sc}$ is continuous
	and $\mu_{sc} \perp m$.
	Then $\mu \Theta = \mu_{ac} \Theta + \mu_{sd} \Theta + \mu_{sc} \Theta$.
	It is straightforward to show that
	$\mu_{ac} \Theta \ll m$, $\mu_{sd} \Theta$ is discrete, $\mu_{sc} \Theta$ is continuous and $\mu_{sc} \Theta \perp m$.
	By Corollary \ref{cor: symmetric conn}, $\mu \Theta = \mu$.
	The uniqueness of the decomposition of measures implies that $\mu_{ac} \Theta = \mu_{ac}, \mu_{sd} \Theta = \mu_{sd}$
	and $\mu_{sc} \Theta = \mu_{sc}$.
	Again, Corollary \ref{cor: symmetric conn} tells us that $\sigma_{ac}, \sigma_{sd}$
	and $\sigma_{sc}$ are symmetric connections.
	The rest of the proof follows from Theorem \ref{thm: connection decom} and Proposition
    \ref{prop: properties of sigma ac}.
\end{proof}
%\\

A decomposition of a symmetric mean as a convex combination of symmetric means
is obtained by normalizing symmetric connections as in the proof of Corollary
\ref{cor: mean decom}.

%\section*{Acknowledgments}
%The author thanks the anonymous authors whose work largely
%constitutes this sample file. He also thanks the INFO-TeX mailing
%list for the valuable indirect assistance he received.

\end{document}